\documentclass[12pt]{article}

\usepackage{amsmath,amsthm,amsfonts,amssymb,amscd,graphicx}

\allowdisplaybreaks[4]
\newtheorem {Lemma}{Lemma}[section]
\newtheorem {Theorem} {Theorem}[section]

\newtheorem {Corollary}{Corollary}[section]
\numberwithin{equation}{section}
\usepackage[figurename=Fig.]{caption}
\usepackage{fullpage}
\allowdisplaybreaks [4]

\begin{document}

\title{A graph for which the second largest distance eigenvalue is less than $\frac{-3+\sqrt{5}}{2}$  is chordal}

\author{Haiyan Guo\textsuperscript{1}\footnote{E-mail: ghaiyan0705@m.scnu.edu.cn},
Bo Zhou\textsuperscript{2}\footnote{
E-mail: zhoubo@scnu.edu.cn}\\
\textsuperscript{1}School of Computer Science, South China Normal University,\\
Guangzhou 510631, P. R. China\\
\textsuperscript{2}School of Mathematical Sciences, South China Normal University,\\
Guangzhou 510631, P. R. China}

\date{}
\maketitle

\begin{abstract}
Let $G$ be a connected graph with vertex set $V(G)$. The distance,
$d_G(u,v)$, between vertices $u$ and $v$ in $G$ is defined as the length of a shortest path between
$u$ and $v$ in $G$. The distance matrix of $G$ is the matrix $D(G)=(d_G(u,v))_{u,v\in V(G)}$. The
second largest distance eigenvalue of $G$ is the second largest one in the spectrum of $D(G)$.
We show that any connected graph with the second largest distance eigenvalue  less than $\frac{-3+\sqrt{5}}{2}$ is chordal, and characterize those bicyclic graphs and split graphs with the second largest distance eigenvalue less than $-\frac{1}{2}$.
\end{abstract}

\section{Introduction}

The adjacency eigenvalues  (commonly called the eigenvalues) of a graph are the eigenvalues of its adjacency matrix.  The second largest eigenvalue of graphs has been widely studied by many mathematicians, see, e.g. \cite[Chapter 4]{BH}, \cite[Subsection 3.11.2]{CRS} and reference therein.
For example,
Cao and Hong \cite{CaoH} characterized the simple graphs with the second largest eigenvalue less than $\frac{1}{3}$. Wu et al. \cite{WQP} determined the simple connected graphs with the second largest eigenvalue less than $\frac{1}{2}$. Cheng et al. \cite{CGK} considered graphs with three eigenvalues and second largest eigenvalue at most $1$. Liu et al. \cite{LCS} determined all connected $\{K_{1,3}, K_5-e\}$-free graphs whose second largest eigenvalue does not exceed $1$. Zhang et al. \cite{ZLK} classified the $2$-partially distance-regular graphs whose each local graph has second largest eigenvalue at most $1$. Recently, multiplicity of the second largest eigenvalue of graphs has also received much attention, see  \cite{CH,HSZZ}.

The distance eigenvalues of a connected graph are the eigenvalues of its distance  matrix.
The distance eigenvalues of graphs were first studied by Graham and Pollak \cite{GP}.  They established a relationship between the number of negative distance eigenvalues of trees and the addressing problem in data communication systems. Thereafter and in particular, in recent 15 years, the distance eigenvalues
attracted much more attention. However, the focus was more on the largest distance eigenvalue (also known as distance spectral radius), see the survey of Aouchiche and Hansen \cite{AH} and \cite{E,HZ,LZ,O,XSW,WZ}.

As far as we know, the only studies on the second largest distance eigenvalue of graphs are as follows. Lin \cite{L} showed that the second largest distance eigenvalue of a graph $G$ is less than the number of triangles in $G$ when the independent number is less than or equal to two, confirming a conjecture in \cite{F}. Xing and Zhou \cite{XZ} characterized all connected graphs with the second largest distance eigenvalue less than $-2+\sqrt{2}$ and all trees with the second largest distance eigenvalue less than $-\frac{1}{2}$. Besides, they also considered unicyclic graphs with a few exceptions whose second largest distance eigenvalue less than $-\frac{1}{2}$.
In \cite{XZ2}, they obtained sharp lower bounds for the second largest distance eigenvalue of the $k$-th power of a connected graph and determined all trees and unicyclic graphs $G$ such that the second largest distance eigenvalue of the squares less than $\frac{\sqrt{5}-3}{2}$. Liu et al. \cite{LXG} proved that the graphs with the second largest distance eigenvalue less than $\frac{17-\sqrt{329}}{2}\approx-0.5692$ are determined by their distance spectra, among other results. Xue et al. \cite{XLS} characterized all block graphs whose second largest distance eigenvalue less than  $-\frac{1}{2}$. Alhevaz et al. \cite{A} gave some upper and lower bounds for the second largest eigenvalue of the generalized distance matrix of graphs in terms of some graph parameters.

A graph is chordal if every cycle  of length at least four  has a chord, where
a chord is an edge joining two non-adjacent vertices of the cycle.
A connected graph on $n$ vertices with $n+1$ edges is called a bicyclic graph. A graph $G$ is a split graph if  both $G$ and $\overline{G}$ are chordal.
In this paper, we show that any connected graphs whose second largest distance eigenvalue is less than $-\frac{1}{2}$ must be  chordal  and  characterize all  bicyclic graphs and split graphs with the second largest distance eigenvalue less than $-\frac{1}{2}$.

\section{Preliminaries}

All graphs considered in this paper are simple and connected. Let $G$ be a graph with vertex set $V(G)$ and edge set $E(G)$.

The distance between vertices $u$ and $v$ in $G$ is defined as the length of a shortest  path connecting $u$ and $v$ in $G$.
Assume that  $V(G)=\{v_1,\dots,v_n\}$ with $n\ge 2$. The distance matrix of $G$ is defined as the $n\times n$ matrix $D(G)=(d_G(v_i,v_j))$.  The eigenvalues of $D(G)$ are called the distance eigenvalues of $G$.
Since $D(G)$ is symmetric, the distance eigenvalues of $G$ are all real, so they may be ordered as
$\lambda_1(G)\ge \dots \ge \lambda_n(G)$.
Then  $\lambda_2(G)$ is the second largest distance eigenvalue of $G$.

For a graph $G$ with  $v\in V(G)$, we use $N_G(v)$ to denote the neighborhood of $v$ in $G$, and let $d_G(v)=|N_G(v)|$ be the degree of $v$ in $G$. For a nonempty vertex subset $S$, let $G[S]$ be the subgraph of $G$ induced by $S$. For a  graph $G$ with $n$ vertices and $m$ edges, if $m=n+c-1$, then $G$ is called a $c$-cyclic graph. Specially, a $c$-cyclic graph with $c=0,1, 2$ is known as  a tree, a unicyclic graph, a bicyclic graph, respectively.

As usual, we denote by $K_n$ the complete graph on $n$ vertices, and $K_{r,s}$ the complete bipartite graph with bipartite sizes $r$ and $s$. Let $S_n=K_{1, n-1}$.  Denote by $C_n$ the cycle on $n$ vertices. Let $\overline{G}$ be the complement of a graph $G$.

A path $u_0\dots u_r$ with $r\geq 1$ in a graph $G$ is called a pendant path of length $r$
at $u_0$ if $\mbox{d}_G(u_0)\geq 3$, the degrees of $u_1,\dots,u_{r-1}$ (if any exists) are all equal to $2$ in $G$, and $\mbox{d}_G(u_r)=1$. In this case, we also say that  $G$ is obtained from $G-\{u_1,\dots, u_r\}$ by attaching a pendant path of length $r$ at $u_0$. For $v\in V(G)$, the graph obtained from $G$ by attaching a pendant path of length $0$ in $G$ is itself.
A pendant path of length $1$ at $u_0$ is called a pendant edge at $u_0$.

A clique of a graph is a set of pairwise adjacent vertices,  and a maximum clique is a clique with maximum cardinality. An independent set is a set of pairwise non-adjacent vertices.

A graph is a split graph if its vertex set can be partitioned into a clique and an independent set. A graph $G$ is a split graph if and only if both $G$ and $\overline{G}$ are chordal, or equivalently,  it does not have an induced $C_4, \overline{C_4}$, or $C_5$.

 A block of a given graph is a maximal connected subgraph  that has no cut vertex.
A connected graph $G$ is called a block graph (also known as clique tree) if all of its blocks are cliques. A block star is a block graph whose blocks contain a common vertex. A block graph $G$ is loose if for each vertex $v\in V(G)$, the number of blocks of $G$ which contain the vertex $v$ is at most $2$.  Let $BG(p,q,3,2,2)$ with $p,q\ge 2$ and $BGA$ be two block graphs as shown in Fig. \ref{ff1}.
\begin{figure}[htbp]
  \centering
  \includegraphics[width=11cm]{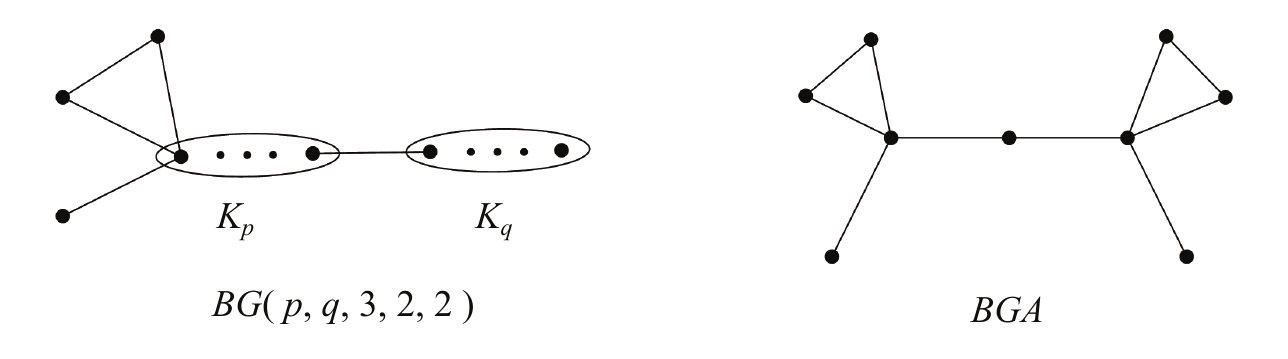}\\
  \caption{Graphs $BG(p,q,3,2,2)$ and BGA.}\label{ff1}
\end{figure}

\begin{Lemma}\cite{XLS}\label{block}
Let $G$ be a block graph with the second largest distance eigenvalue $\lambda_2(G)$. Then $\lambda_2(G)<-\frac{1}{2}$ if and only if
\begin{itemize}
   \item $G$ is a block star, or
   \item $G$ is a loose block graph, or
   \item $G$ is a nontrivial  connected induced subgraph of $BG(p, q, 3, 2, 2)$, or
   \item $G$ is a nontrivial connected induced subgraph of $BGA$.
  \end{itemize}
\end{Lemma}

Let $\mathcal{G}$ denote the set of unicyclic graphs with at least $5$ vertices obtained from $C_3=uvw$ by attaching a pendant path at $u,v,w$, respectively.  Denote by $S_n^+$  the $n$-vertex unicyclic graph obtained by adding an edge to the star $S_n$, where $n\ge 3$.

\begin{Lemma}\cite{XZ}\label{XZ}
Let $G$ be a unicyclic graph, where $G\notin \mathcal{G}$. Then $\lambda_2(G)<-\frac{1}{2}$ if and only if $G\cong S_n^+$ for $n\ge3$, or one of the six graphs $U_1,\dots, U_6$ shown in Fig. \ref{ff2}.
\end{Lemma}

Note that any $G$ in $\mathcal{G}$ is a loose block graph, so  we have $\lambda_2(G)<-\frac{1}{2}$ by Lemma \ref{block}. Combining  Lemmas \ref{block} and \ref{XZ}, we have the following.

\begin{Corollary}\label{un}
Let $G$ be a unicyclic graph. Then $\lambda_2(G)<-\frac{1}{2}$ if and only if $G\cong S_n^+$ for $n\ge3$, $G\in \mathcal{G}$, or one of the six graphs $U_1, \dots, U_6$ shown in Fig. \ref{ff2}.
\end{Corollary}
\begin{figure}[htb]
  \centering
  \includegraphics[width=10cm]{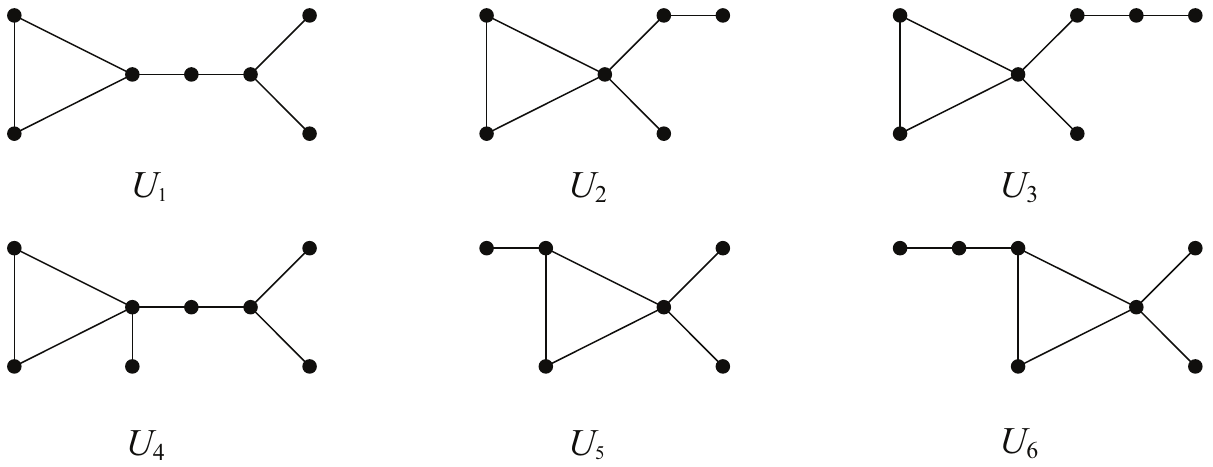}\\
  \caption{Graphs $U_1, \dots, U_6$.}\label{ff2}
\end{figure}

For an $n\times n$ real symmetric matrix $M$, let $\rho_1(M)\ge \ldots\ge\rho_n(M)$ be the eigenvalues of $M$.  The following lemma is the classical Cauchy's interlacing theorem, see \cite[Theorem 4.3.28]{HJ} or \cite{H}.

\begin{Lemma}[Cauchy's interlacing theorem] \label{Interlacing}
Let $A$ be an $n\times n$  symmetric matrix. If $B$ is an $m\times m$ principal submatrix of $A$, then $\rho_i(A)\ge \rho_i(B)\ge\rho_{n-m+i}(A)$ for $1\le i\le m$.
\end{Lemma}

If $H$ is a connected induced subgraph of $G$ and $d_H(u,v)=d_G(u,v)$ for all $\{u,v\}\subseteq V(H)$, then we say $H$ is a distance-preserving subgraph of $G$. In this case, $D(H)$ is a principal submatrix of $D(G)$, so $\lambda_2(G)\ge\lambda_2(H)$ by Lemma \ref{Interlacing}. Specially, if $G$ is a connected graph with $\lambda_2(G)<-\frac{1}{2}$, then $\lambda_2(H)<-\frac{1}{2}$ for any  distance-preserving subgraph $H$ of $G$. In this paper, we are concerned with the graphs $G$ with $\lambda_2(G)<-\frac{1}{2}$.
So, we call a graph $H$ a forbidden subgraph of $G$ if $H$ is a distance-preserving subgraph of $G$ but
$\lambda_2(H)\ge-\frac{1}{2}$. We show some forbidden subgraphs $F_1, \dots, F_{13}$ in Fig. \ref{ff3} that we need in the proofs,  where the second largest distance eigenvalue is also listed below the corresponding graph.
\begin{figure}[ht]
  \centering
  \includegraphics[width=14cm]{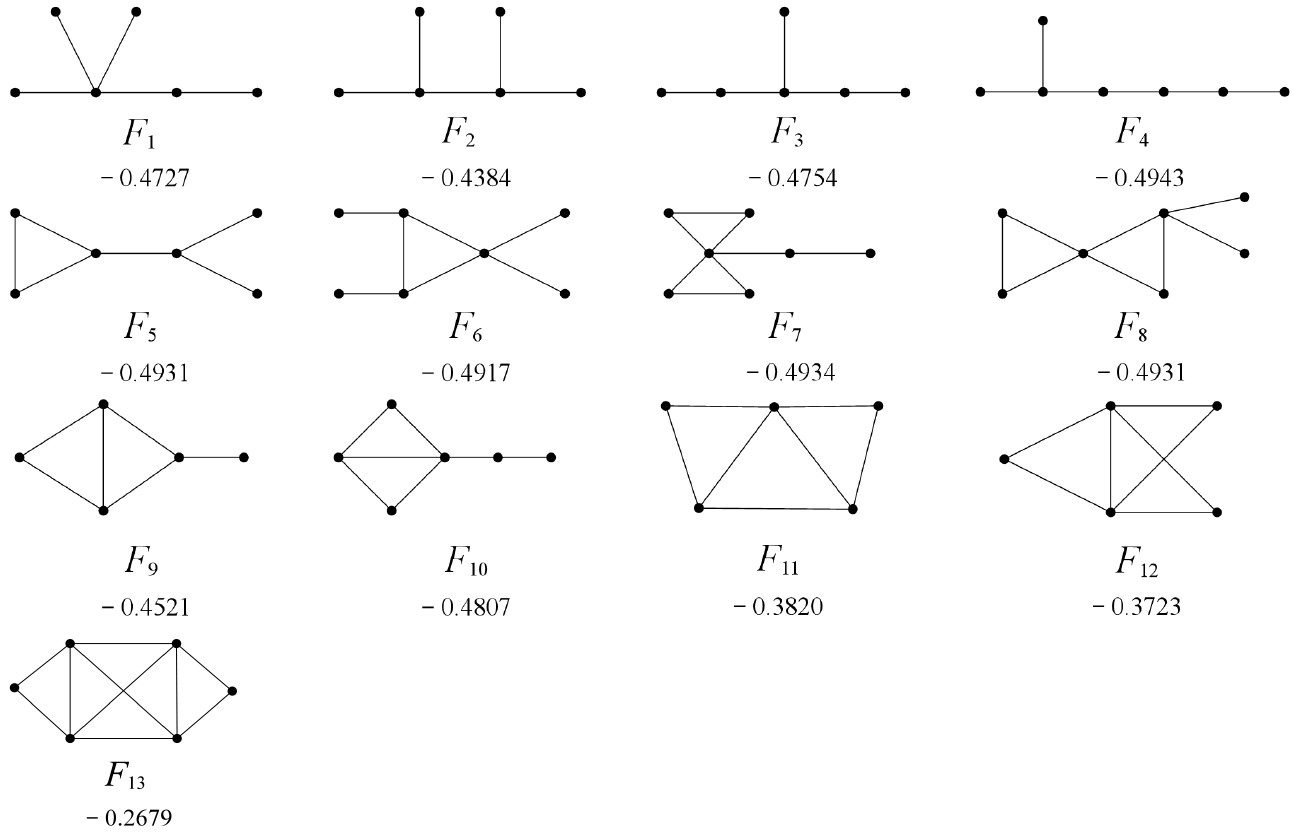}\\
  \caption{Forbidden subgraphs $F_1, \dots, F_{13}$.}\label{ff3}
\end{figure}

The following lemma may be checked easily.

\begin{Lemma}\label{sc} \cite{GJS}
For the cycle $C_n$ on $n\ge 3$ vertices,
\[\lambda_2(C_n)=
  \begin{cases}
  0 & \text{if $n$ is even,}\\
  -\frac{1}{4}\sec^2\frac{\pi}{n} & \text{if $n$ is odd.}
  \end{cases}
\]
\end{Lemma}


Let $J_{s\times t}$ be the $s\times t$ matrix of all $1$'s, and $I_s$ the identity matrix of order $s$. For convenience, let $J_s=J_{s\times s}$.

\section{Any graph $G$ with $\lambda_2(G)<-\frac{1}{2}$ must be chordal}

In this section, we show that any  graph whose  second largest distance eigenvalue is less than $-\frac{1}{2}$ must be chordal.

\begin{Theorem}\label{chordal}  Suppose that $G$ is a graph with $\lambda_2(G)<-\frac{1}{2}$, then $G$ is chordal.
\end{Theorem}

\begin{proof} We prove the theorem by contradiction. Suppose that it is not true. Then there exists a graph $G$  with $\lambda_2(G)<-\frac{1}{2}$ but $G$ is not chordal. So there is an induced cycle in $G$ with length at least four.  We choose such an induced cycle $C_k$ of $G$ so that its length $k$ is minimum. For odd $k\ge 5$, $-\frac{1}{4}\sec^2\frac{\pi}{k}$ is strictly increasing. By Lemmas \ref{Interlacing} and \ref{sc},  $C_k$ is not a distance-preserving subgraph of $G$, and in particular, $k\ge 6$.
So there are two non-adjacent vertices, say $u$ and $v$,  on $C_k$ such that $d_{C_k}(u,v)>d_G(u,v)$. Let
\[
\mathcal{C}=\{\{u,v\}: d_{C_k}(u,v)>\max\{d_{G}(u,v), 1\}, u,v\in V(C_k)\}.
\]
Assume that $\{x,y\}\in \mathcal{C}$ so that $d_{G}(x,y)$ is minimum.
Let $s=d_G(x,y)$ and $t=d_{C_k}(x,y)$.
Since $C_k$ is an induced cycle of $G$, we have $2\le s<t\le \frac{k}{2}$.

Let  $P=x_0x_1\dots x_s$ be a shortest path between $x$ and $y$ in $G$, where $x=x_0$ and $y=x_s$.
Let  $Q=y_0 y_1\dots y_{t}$ be the path on $C_k$ from $x$ to $y$ with length $t$, where  $x=y_0$ and $y=y_t$.
Let $G'=G[V(P)\cup V(Q)]$. Note that  $G'$ contains a Hamiltonian cycle $C=x_0x_1\dots x_sy_{t-1}\dots y_1x_0$.

\noindent \textbf{Claim 1.} $G'$ is  chordal.

Note that  $|V(G')|=s+t<2t\le k$. Thus, if $G'$ contains  an induced cycle, then its length  is  less than $k$,  so its length must be three by the choice of $C_k$.  Thus  $G'$ is  chordal.  Claim 1 follows.

By the choice of $C_k$ and $P$,
any chord of $C$ joins a vertex  from $V(Q)\setminus\{x,y\}=\{y_1,\dots, y_{t-1}\}$ and a vertex  from $V(P)\setminus\{x,y\}=\{x_1,\dots,x_{s-1}\}$.

\noindent \textbf{Claim 2.} $x_1y_1\in E(G')$.

If $x_1y_1\notin E(G')$, then there is an induced cycle of length at least four (and less than $k$) containing $x$ in $G'$, contradicting Claim 1. This proves Claim 2.

By Claim 2, $G'-x_0$ contains a Hamiltonian cycle $C'=x_1\dots x_sy_{t-1}\dots y_1x_1$.
By Claims 1 and 2,  $G'-x_0$ is chordal.

\noindent \textbf{Case 1.} $C'$ has a chord incident with $x_1$.

Since $G'-x_0$ is  chordal, it has no induced cycle of length at least four, so $x_1y_2\in E(G')$. If $x_1y_3\in E(G')$, then
$G'[\{x_0,x_1,y_1,y_2,y_3\}]\cong F_{11}$, which is a forbidden subgraph, a contradiction. This shows that $t\ge 4$ and  $x_1y_3\notin E(G')$.
Then, step by step, we have  $x_1y_i\notin E(G')$ for $i=4,\dots, t-1$, as, otherwise, there is an induced cycle of length at least four in $G'$, contradicting Claim 1. If $s=2$, then $y_2y_3\dots y_tx_1y_2$ is an induced cycle of length at least four in $G'$, contradicting Claim 1. So $s\ge 3$.

\noindent \textbf{Case 1.1} $x_2y_1\in E(G')$.

If $x_2y_2\notin E(G')$, then $G'[\{x_0,x_1,x_2,y_1,y_2\}]\cong F_{12}$, which is a forbidden subgraph,  a contradiction. So $x_2y_2\in E(G')$.

If $x_2y_3\in E(G')$, then $G'[\{x_0,x_1,x_2,y_1,y_2,y_3\}]\cong F_{13}$, which is a forbidden subgraph,  also a contradiction. So $x_2y_3\notin E(G')$. Step by step, we have  $x_2y_i\notin E(G')$ for $i=4,\dots, t-1$ if $t\ge 5$. So $x_2y_i\notin E(G')$ for $3\le i\le t-1$.

Next, we claim that $x_3y_i\notin E(G')$ for $1\le i\le t-1$ if $s\ge 4$. Otherwise,
the subgraph of $G'[\{x_0,x_1,x_2,x_3,y_1\}]\cong F_{11}$ if $x_3y_1\in E(G')$,
the subgraph of $G'[\{x_0,x_1,x_2,x_3,y_1,y_2\}]\cong F_{13}$ if $x_3y_2\in E(G')$, and there is an induced cycle of length at least four in $G'$ if $x_3y_i\in E(G')$ for $3\le i\le t-1$, so we have a contradiction in any case.
Step by step, we have $x_iy_j\notin E(G')$ for $4\le i\le s-1$ and $1\le j\le t-1$.  So $G'[V(G')\setminus \{x_0,x_1,y_1\}]$ is an induced cycle of $G'$  with length at least four, a contradiction.

\noindent \textbf{Case 1.2} $x_2y_1\notin E(G')$.

In this case, $x_2y_2\notin E(G')$. Otherwise, $G'[\{x_0,x_1,x_2,y_1,y_2\}]\cong F_{11}$, which is a forbidden subgraph, a contradiction. Step by step, we have $x_2y_i\notin E(G')$ for $3\le i\le t-1$,  and $x_iy_j\notin E(G')$ for $3\le i\le s-1$ and $1\le j\le t-1$ if $s\ge 4$. So $G'[V(G')\setminus \{x_0,y_1\}]$ is an induced cycle of $G'$ with length at least four, a contradiction.

\noindent \textbf{Case 2.} $C'$ has  no chord incident with $x_1$.

Suppose that $t=3$. Then $s=2$. In the case of $t=3$ or $s=2$, $G'[V(G')\setminus\{x_0\}]$ is an induced cycle with length at least four, a contradiction. So $s\ge 3$ and $t\ge 4$.

Since $G'-x_0$ is chordal, $x_2y_1\in E(G')$.

We claim that $x_2y_i\notin E(G')$ for $2\le i\le t-1$. Otherwise, $G'[\{x_0,x_1,x_2,y_1,y_2\}]\cong F_{11}$ if $x_2y_2\in E(G')$, and there is an induced cycle of length at least four if $x_2y_i\in E(G')$ for $3\le i\le t-1$, so we have a contradiction in either case. Step by step, we have  $x_iy_j\notin E(G')$ for $3\le i\le s-1$ and $1\le j\le t-1$. So  $G'[V(G')\setminus \{x_0,x_1\}]$ is an induced cycle of length at least four, a contradiction.

Combing case 1 and 2, we complete the proof.
\end{proof}

By the above argument of Theorem \ref{chordal}, we actually prove the following strong result.

\begin{Theorem}\label{cho2}  Suppose that $G$ is a graph with $\lambda_2(G)<\frac{-3+\sqrt{5}}{2}$, then $G$ is chordal.
\end{Theorem}

\begin{proof} By a direct calculation, $\lambda_2(C_5)=\frac{-3+\sqrt{5}}{2}=\lambda_2(F_{11})$  is a root of the  equation $x^2+3x+1=0$.
 Replacing $-\frac{1}{2}$ in the proof of Theorem \ref{chordal} by $\frac{-3+\sqrt{5}}{2}$ and calling
a distance-preserving subgraph $H$ of $G$ forbidden if $\lambda_2(H)\ge \frac{-3+\sqrt{5}}{2}$, we complete the proof.
\end{proof}

\section{Bicyclic graphs $G$ with $\lambda_2(G)<-\frac{1}{2}$}

For integers $p$, $q$ and $s$ with $p\ge 3$, $q\ge3$ and $s\ge0$, let $\infty(p, q, s)$ be the bicyclic graph obtained from the cycles $C_p=u_1u_2 \dots u_pu_1$, $C_q=v_1 v_2\dots v_q v_1$ and the path $P_{s+1}=p_0p_1\dots p_s$ by identifying $u_1$ with $p_0$ and identifying $v_1$ with $p_s$. In particular, $\infty(p, q, 0)$ consists of two cycles of lengths $p$ and $q$ respectively with precisely one vertex in common.

For positive integers $p$, $q$ and $s$, where at most one of $p$, $q$, $s$ is equal to one,  let $\theta(p,q,s)$ be the bicyclic graph obtained from three paths $P_{p+1}=u_0u_1\dots u_{p-1}u_{p}$, $P_{q+1}=v_0v_1\dots v_{q-1} v_{q}$ and $P_{s+1}=w_0w_1\dots w_{s-1}w_{s}$ by identifying $u_0$, $v_0$ and $w_0$ to a new vertex $x$ and identifying $u_{p}$, $v_{q}$ and $w_{s}$ to a new vertex $y$.

Graphs $\infty(p,q,s)$ and $\theta(p,q,s)$ are depicted in Fig. \ref{ff4}.
\begin{figure}[htbp]
  \centering
  \includegraphics[width=13cm]{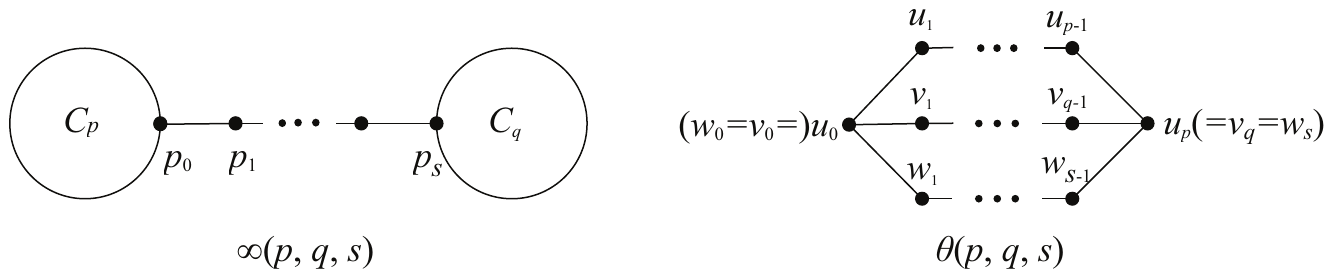}\\
  \caption{Graphs $\infty(p,q,s)$ and $\theta(p, q, s)$.}\label{ff4}
\end{figure}

For a bicyclic graph $G$, if $\infty(p, q, s)$ is an induced subgraph of $G$ for some $p,q,s$, then we say $G$ is a $\infty$-bicyclic graph. Otherwise, $G$ contains $\theta(p', q', s')$ as an induced subgraph for some $p',q',s'$, so we say  $G$ is a $\theta$-bicyclic graph.

To state the results, we define several families of bicyclic graphs.

Let $u_1$, $u_2$, $u_3$, $u_4$ be the four vertices with degree two in the cycles of $\infty(3,3,s)$ with $u_1u_2,u_3u_4\in E(\infty(3,3,s))$, where $s\ge 0$. We use $B(s;h_1,h_2,h_3,$ $h_4)$ to denote the graph obtained from $\infty(3,3,s)$ by attaching a pendant path of length $h_i$ at $u_{i}$, respectively, where $h_i\ge0$ for $1\le i\le4$, see Fig. \ref{ff5}.

\begin{figure}[htbp]
  \centering
  \includegraphics[width=12cm]{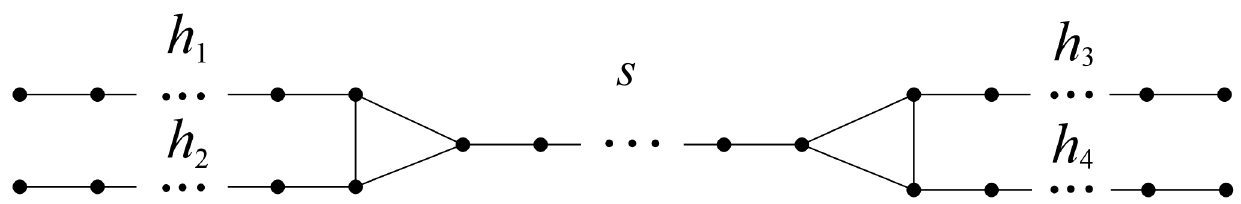}\\
  \caption{Graph $B(s;h_1,h_2,h_3,h_4)$.}\label{ff5}
\end{figure}
Let $B^{\infty}_t$ be the graph obtained from $\infty(3,3,0)$ by attaching $t$ pendant edges at the vertex with maximum degree, where $t\ge 0$, as depicted in Fig. \ref{ff6} (left).
We use $B^{\theta}_k$ to denote the graph obtained from $\theta(2,2,1)$ by attaching $k$ pendant edges at a vertex of degree three in $\theta(2,2,1)$, where $k\ge 0$, as depicted in Fig. \ref{ff6} (right).

\begin{figure}[htbp]
  \centering
  \includegraphics[width=9cm]{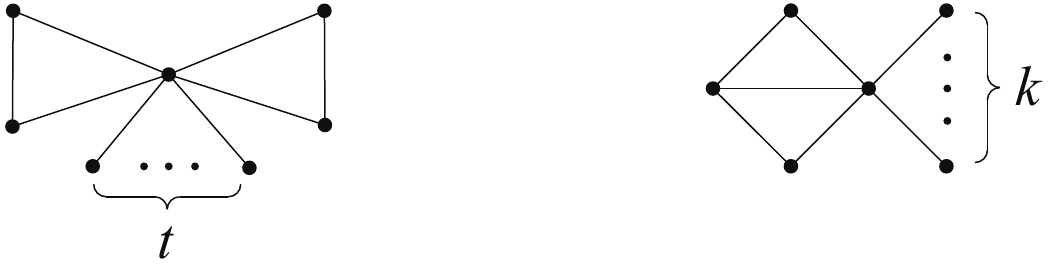}\\
  \caption{Graphs $B^{\infty}_t$ and $B^{\theta}_k$.}
  \label{ff6}
\end{figure}

\begin{Theorem} \label{MAIN}
Let $G$ be a bicyclic graph. Then $\lambda_2(G)<-\frac{1}{2}$ if and only if
\begin{itemize}
   \item $G\cong B_1, \dots, B_7$ as displayed in Fig. \ref{ff7}, or
   \item $G\cong B(s;h_1,h_2,h_3,h_4)$, where $s\ge 0$ and $h_i\ge0$ for $1\le i\le4$, or
   \item $G\cong B^{\infty}_t$ with $t\ge0$, or
   \item $G\cong B^{\theta}_k$, where $k\ge 0$.
  \end{itemize}

\begin{figure}[ht]
  \centering
  \includegraphics[width=12cm]{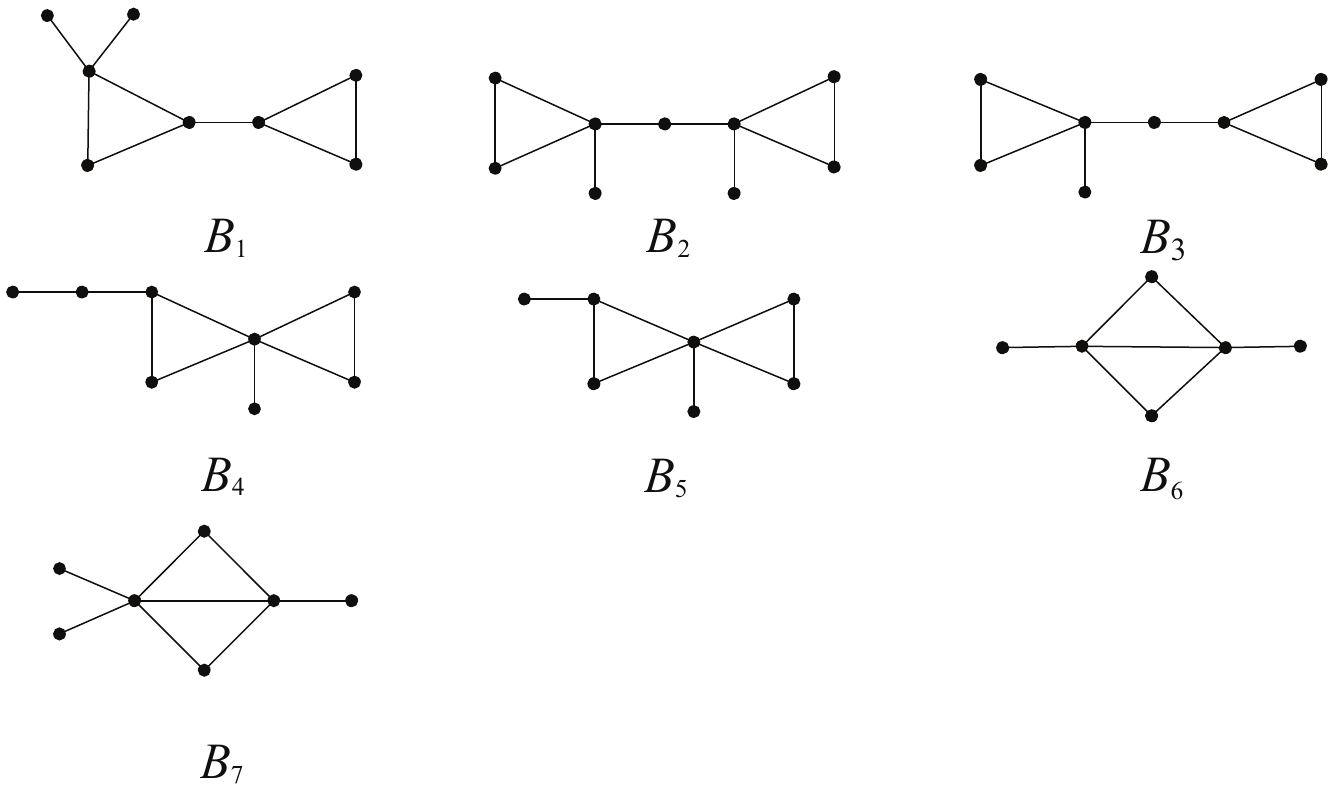}\\
  \caption{ Graphs $B_1, \dots, B_7$.}\label{ff7}
\end{figure}
\end{Theorem}

Note that a bicyclic graph is either a $\infty$-bicyclic graph or a $\theta$-bicyclic graph.
To prove Theorem \ref{MAIN}, it suffices to show the following two lemmas.

\begin{Lemma}\label{bi1}
Let $G$ be a $\infty$-bicyclic graph. Then $\lambda_2(G)<-\frac{1}{2}$ if and only if $G\cong B_1,\dots,B_5$, $G\cong B(s;h_1,h_2,h_3,h_4)$, where $s\ge 0$ and $h_i\ge0$ for $1\le i\le4$, or $G\cong B^{\infty}_t$ with $t\ge0$.
\end{Lemma}

\begin{Lemma}\label{bi2}
Let $G$ be a $\theta$-bicyclic graph. Then $\lambda_2(G)<-\frac{1}{2}$ if and only if $G\cong B_6, B_7$, or $G\cong B^{\theta}_k$, where $k\ge 0$.
\end{Lemma}

Firstly, we give the proof of Lemma \ref{bi1}.

\begin{proof}[Proof of Lemma \ref{bi1}] Let $G$ be a $\infty$-bicyclic graph with $\lambda_2(G)<-\frac{1}{2}$. By Theorem \ref{chordal},  $G$ is chordal. So the two cycles of $G$ are both of length three.
Suppose without loss of generality that $C_p=u_1u_2u_3u_1$ and $C_q=v_1 v_2v_3v_1$ are the two cycles of $G$, and $P_{s+1}=p_0p_1\dots p_s$ is the shortest path connecting a vertex of $C_p$ and a vertex of $C_q$, say $u_1=p_0$ and $p_s=v_1$.

\noindent \textbf{Case 1.} $s\ge2$.

In this case, we have $d_G(p_i)=2$ for $1\le i\le s-1$. Otherwise  $F_3$ is an induced distance-preserving subgraph of $G$, contradicting that $F_3$ is a forbidden subgraph of $G$.  Suppose without loss of generality that $d_G(u_1)\ge d_G(v_1)$. Since $F_1$ is a forbidden subgraph, we have $d_G(u_1)\le 4$.
That is, $d_G(u_1)=3,4$.

Suppose first that  $d_G(u_1)=4$.  Then $s\le 2$ since $F_4$ is a forbidden subgraph. So $s=2$. Let $w$ be the unique vertex in  $N_G(u_1)\setminus\{u_2,u_3,p_1\}$. We have $d_G(w)=1$ as $F_3$ is a forbidden subgraph. Similar argument leads to $d_G(u_2)=d_G(u_3)=2$. Since $F_4$ is a forbidden subgraph, we have $d_G(v_2)=d_G(v_3)=2$. Thus, if $d_G(v_1)=4$, then $G\cong B_2$, and if $d_G(v_1)=3$, then $G\cong B_3$.

Suppose next that $d_G(u_1)=3$. As $d_G(u_1)\ge d_G(v_1)\ge 3$, one gets $d_G(v_1)=3$.  Since $F_4$ is a forbidden subgraph, we have $\max\{d_G(u_2),d_G(u_3)\}\le 3$, and if $d_G(u_i)=3$ for $i=2, 3$, then there is a pendant path at $u_i$.  Similarly, $\max\{d_G(v_2),d_G(v_3)\}\le 3$,  and if $d_G(v_i)=3$ for $i=2, 3$, then there is a pendant path at $v_i$. Hence, $G\cong B(s;h_1,h_2,h_3,h_4)$, where $s\ge 2$ and $h_i\ge0$ for $1\le i\le4$.

\noindent \textbf{Case 2.} $s=1$.

Since $F_5$ is a forbidden subgraph, we have $d_G(u_1)=d_G(v_1)=3$. Assume that $\max\{d_G(u_2),$ $d_G(u_3),d_G(v_2),d_G(v_3)\}=d_G(u_2)$. Since $F_1$ is a forbidden subgraph, we have $d_G(u_2)\le4$.

Suppose that $d_G(u_2)=4$. Let  $N_G(u_2)\setminus\{u_1,u_3\}=\{x_1,x_2\}$. We have $d_G(x_1)=d_G(x_2)=1$ as $F_3$ is a forbidden subgraph. Since $F_6$ is a forbidden subgraph, we have $d_G(u_3)=2$. Since $F_4$ is a forbidden subgraph, we have $d_G(v_2)=d_G(v_3)=2$. Hence, $G\cong B_1$.

If $d_G(u_2)=3$, then for any $w\in V(G)\setminus\{u_1,u_2,u_3,v_1,v_2,v_3\}$, one has $d_G(w)=1, 2$ since $F_4$ is a forbidden subgraph. Hence, $G\cong B(1;h_1,h_2,h_3,h_4)$, where $h_1\ge1$ and $h_i\ge 0$ for $2\le i\le4$.

If $d_G(u_2)=2$, then $G\cong \infty(3,3,1)\cong B(1;0,0,0,0)$.

\noindent \textbf{Case 3.} $s=0$.

If $d_G(u_1)=4$, then $\max\{d_G(u_2),d_G(u_3),d_G(v_2),d_G(v_3)\}\le3$ as $F_5$ is a forbidden subgraph, and  every vertex not on the cycles is of degree one or two due to the fact that $F_4$ and $F_5$ are both  forbidden subgraphs. This implies that $G\cong B(0;h_1,h_2,h_3,h_4)$, where $h_i\ge 0$ for $1\le i\le 4$.

Suppose that $d_G(u_1)\ge 5$. Let $N_G(u_1)\setminus\{u_2,u_3,v_2,v_3\}=\{x_1,\ldots, x_t\}$, where $t\ge 1$. Since $F_7$ is a forbidden subgraph, one has $d_G(x_1)=\dots=d_G(x_t)=1$. Moreover, $\max\{d_G(u_2),d_G(u_3),d_G(v_2),d_G(v_3)\}\le 3$ as $F_5$  is a forbidden subgraph.

If $d_G(u_1)\ge 6$, then, since $F_1$ is a forbidden subgraph, we have $d_G(u_2)=d_G(u_3)=d_G(v_2)=d_G(v_3)=2$, which implies that $G\cong B^{\infty}_t$, where $t\ge2$.

Suppose that $d_G(u_1)=5$.  Assume that $\max\{d_G(u_2),d_G(u_3),d_G(v_2),d_G(v_3)\}$ $=d_G(u_2)$. Since $F_5$ is a forbidden subgraph, we have $d_G(u_2)\le 3$. Then, if $d_G(u_2)=2$, then $G\cong B^{\infty}_1$. Suppose that $d_G(u_2)=3$. Since $F_3$ is a forbidden subgraph, we have $d_G(v_2)=d_G(v_3)=2$. As $F_6$ is a forbidden subgraph, we have $d_G(u_3)=2$. Let $w_1$ denote the neighbor of $u_2$ not on the cycles. Since $F_5$ is a forbidden subgraph, we have $d_G(w_1)\le 2$. Then $G\cong B_5$ if $d_G(w_1)=1$.
If $d_G(w_1)=2$, then denoting by $w_2$ the neighbor of $w_1$ different from $u_2$, one has $d_G(w_2)=1$ as $F_4$ is a forbidden subgraph, which implies that $G\cong B_4$. Hence, $G\cong B_4,B_5$, or $B^{\infty}_1$.

Combining Cases 1--3, we have  $G\cong B_1, \dots, B_5$, or $G\cong B(s;h_1,h_2,h_3,h_4)$, where $s\ge 0$ and $h_i\ge0$ for $1\le i\le4$, or $G\cong B^{\infty}_t$ with $t\ge0$.

Conversely, suppose that
$G\cong B_1, \dots, B_5$, or $G\cong B(s;h_1,h_2,h_3,h_4)$, where $s\ge 0$ and $h_i\ge0$ for $1\le i\le4$, or $G\cong B^{\infty}_t$ with $t\ge0$.
By a direct calculation, we have $\lambda_2(B_1)\approx-0.5110<-\frac{1}{2}$, $\lambda_2(B_4)\approx-0.5023<-\frac{1}{2}$. Since $B_5$ is an induced distance-preserving subgraph of $B_4$. By Lemma \ref{Interlacing}, we have $\lambda_2(B_5)\le \lambda_2(B_4)<-\frac{1}{2}$.
Note that $B_2\cong BGA$ and $B_3$ is an induced distance-preserving subgraph of $BGA$. By Lemma \ref{block}, we have $\lambda_2(B_3)\le \lambda_2(B_2)<-\frac{1}{2}$.
If $G\cong B(s;h_1,h_2,h_3,h_4)$, where $s\ge 0$ and $h_i\ge0$ for $1\le i\le4$,
then $G$ is a loose block graph; If $G\cong B^{\infty}_t$ with $t\ge0$, then $G$ is a block star. In either case, we have by Lemma \ref{block} that $\lambda_2(G)<-\frac{1}{2}$.
\end{proof}

Next, we move to give proof of Lemma \ref{bi2}.

\begin{proof}[Proof of Lemma \ref{bi2}]  By Theorem \ref{chordal}, $G$ is chordal, so $\theta(2, 2, 1)$ is an induced subgraph of $G$. Let $u_0$ and $u_2$ be the two vertices of  $\theta(2, 2, 1)$ with degree three, and $u_1$ and $v_1$ the two vertices of  $\theta(2, 2, 1)$ with degree two.
Since $F_9$ is a forbidden subgraph, one gets $d_G(u_1)=d_G(v_1)=2$. Since $F_{10}$ is a forbidden subgraph, there can only be  some pendant edges at $x$ or $y$. Assume that $d_G(x)\ge d_G(y)$. Then  $d_G(y)\le4$, as otherwise, there would be a forbidden subgraph $F_2$.

If $d_G(y)=4$, then, since $F_1$ is a forbidden subgraph, we have $d_G(x)\le 5$, so, $G\cong B_6, B_7$. If $d_G(y)=3$, then $G\cong B^{\theta}_k$, where $k\ge 0$.

Conversely, suppose that $G\cong B_6, B_7$, or $G\cong B_k^{\theta}$ with $k\ge0$.
 By a direct calculation, we have $\lambda_2(B_6)\approx-0.5578<-\frac{1}{2}$ and  $\lambda_2(B_7)\approx-0.5119<-\frac{1}{2}$.  Assume that $G=B_k^{\theta}$ with $k\ge0$. Then $|V(G)|=k+4$. If $k=0$, then $G\cong \theta(2,2,1)$. By a direct calculation, we have $\lambda_2(\theta(2,2,1))\approx-0.5616<-\frac{1}{2}$. Assume that $k\ge 1$. Then $n\ge 5$. Let $V_1$ be the set of two vertices of degree $2$ and $V_2$ be the set of vertices of degree one. So we partition  $V(G)$ as $V(G)=\{w\}\cup\{z\}\cup V_1\cup V_2$, where $w$ is the vertex with maximum degree and $z$ is the vertex with degree $3$. Under this partition, we have
\[D(G)+2I_n=
\left(
\begin{array}{cccc}
 2  & 1 & J_{1\times2} & J_{1\times(n-4)} \\
 1 & 2 & J_{1\times2} & 2J_{1\times(n-4)} \\
 J_{2\times1} & J_{2\times1} & 2J_2& 2J_{2\times(n-4)} \\
 J_{(n-4)\times1} & 2J_{(n-4)\times1} & 2J_{(n-4)\times2} & 2J_{n-4} \\
\end{array}
\right).
\]
It is easily seen that $D(G)+2I_n$ is of rank $4$, which implies that $0$ is an eigenvalue of $D(G)+2I_n$ with multiplicity $n-4$.
Note that the above partition for $D(G)+2I_n$ is equitable, thus the eigenvalues of its quotient matrix $Q$ are also the eigenvalues of $D(G)+2I_n$, see \cite[Lemma 2.3.1]{BH}, where
\[Q=\left(
\begin{array}{cccc}
 2 & 1 & 2 & n-4 \\
 1 & 2 & 2 & 2(n-4) \\
 1 & 1 & 4 & 2(n-4) \\
 1 & 2 & 4 & 2(n-4) \\
\end{array}
\right).\]
Let $f(\lambda)=\det(\lambda I_4-Q)$. By a direct calculation,
\[
f(\lambda)=\lambda^4-2n\lambda^3+3(n+1)\lambda^2+4(n-6)\lambda-6n+24.
 \]
Note that $f(+\infty)>0$, $f(7)=2404-517n<0$, $f(\frac{3}{2})=-\frac{3}{16}<0$, $f(0)=24-6n<0$, $f(-\infty)>0$, $f(\frac{3n-1}{2n})=\frac{1}{16n^4}(n^4-12n^3+70n^2-12n+1)$. Let $g(x)=x^4-12x^3+70x^2-12x+1$. Then $g'(x)=4(x^3-9x^2+35x-3)$, $g''(x)=4(3x^2-18x+35)$. Since $g''(x)>0$ for all $x$, $g'(x)\ge g'(5)=288>0$ for $x\ge 5$, which implies that $g(x)\ge g(5)=816>0$. Thus $f(\frac{3n-1}{2n})>0$. It follows that the second largest  root $\lambda^{(2)}$ of $f(\lambda)=0$ satisfies
$\frac{3n-1}{2n}<\lambda^{(2)}<\frac{3}{2}$. By the above argument, $\lambda^{(2)}$ is the second largest eigenvalues of $D(G)+2I_n$, i.e., $\lambda_2(G)<-\frac{1}{2}$.
\end{proof}

From the proof of Lemma \ref{bi2}, we have $-\frac{n+1}{2n}<\lambda_2( B_k^{\theta})<-\frac{1}{2}$. Thus, the limit of $\lambda_2( B_k^{\theta})$ as $n$ approaches $+\infty$ is $-\frac{1}{2}$.

\section{Split graphs $G$ with $\lambda_2(G)<-\frac{1}{2}$}

In the following,  we view a clique $K$ of cardinality $s$ of a graph $G$ as $K_s$, the  subgraph of $G$ induced by $K$.

For nonnegative integer $t$, let $SP^{t}$ be the split graph consisting of a clique $K_4$ and an independent set $I=\{x_1,\dots,x_t, w\}$ so that $x_1,\dots, x_t$ have a unique neighbor $u\in V(K_4)$ and $w$ has exactly two neighbors $u,v\in V(K_4)$, see Fig. \ref{ff8}. In particular,  $SP^{0}$ is the split graph with a clique $K_4$ and an  independent set $I=\{w\}$ so that $w$ has exactly two neighbors in $V(K_4)$.

\begin{figure}[htb]
  \centering
  \includegraphics[width=4cm]{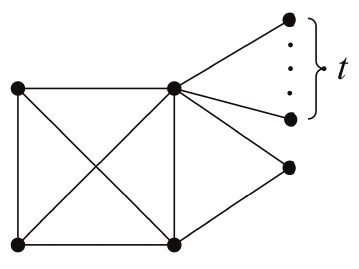}\\
  \caption{Graph $SP^t$.}\label{ff8}
\end{figure}

Let $s$ be an integer with $s\ge 2$.
Let $K_s(x_1,\dots,x_r)$ be the split graph obtained from $K_s$ with vertex set $\{v_1,\dots,v_s\}$ by attaching $x_i$ pendant edges at $v_i$ for $i=1,\dots, r$, where $1\le r\le s$. In particular, $K_2(1,1)\cong P_4$ and $K_2(t)\cong S_{t+2}$.

\begin{Theorem} \label{SED}  Let $G$ be a split graph. Then $\lambda_2(G)<-\frac{1}{2}$ if and only if
\begin{itemize}
   \item $G\cong SP_1, SP_2$ as displayed in Fig. \ref{ff9}, or
   \item $G\cong  B_6, B_7$, or
   \item $G\cong SP^t$, where $t\ge 0$, or
   \item $G\cong B^{\theta}_k$, where $k\ge 0$, or
   \item $G\cong K_s(t)$, where $s\ge 2$ and $t\ge0$, or
   \item $G\cong K_s(a,1)$, where $s\ge 2$ and $a=1,2$, or
   \item $G\cong K_{s}(\underbrace{1,\dots, 1}_t)$, where $s\ge 3$ and $t\ge 3$.
  \end{itemize}
\end{Theorem}

\begin{figure}[htb]
  \centering
  \includegraphics[width=10cm]{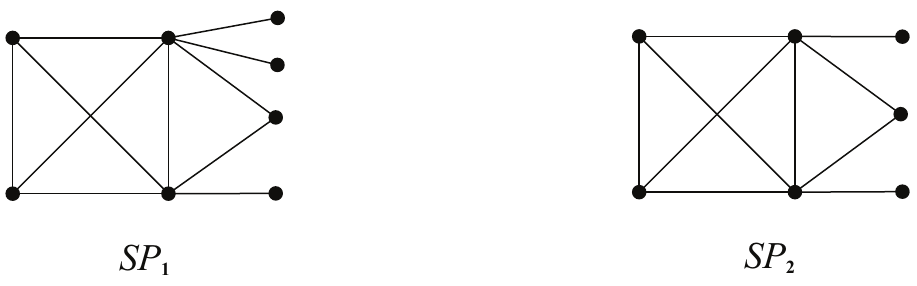}\\
  \caption{Graphs $SP_1, SP_2$.}\label{ff9}
\end{figure}

To prove Theorem \ref{SED}, we need two lemmas.

\begin{Lemma}\label{b5} Let $G=SP^{t}$ with $t\ge 0$. Then $\lambda_2(G)<-\frac{1}{2}$.
\end{Lemma}
\begin{proof} Note that $SP^{t}$ is a distance-preserving induced subgraph of $SP^{t+1}$. From Lemma \ref{Interlacing}, we have $\lambda_2(SP^{t})\le \lambda_2(SP^{t+1})$, which implies that  $\{\lambda_2(SP^{t}): t=0,1,\dots\}$ is a non-decreasing sequence. So it suffices to show that $\lambda_2(G)<-\frac{1}{2}$ for large enough $t$.

Let $G=SP^{t}$.  Then $|V(G)|=t+5$. We partition $V(G)$ as $V(G)=V_1\cup V_2\cup\{u\}\cup\{v\}\cup\{w\}$, where $V_1=K_4\setminus\{u,v\}$, $V_2=\{x_1,\dots,x_t\}$.
Under this partition, we have
\[D(G)=
\left(
\begin{array}{ccccc}
 J_{2}-I_{2}  & 2J_{2\times t} & J_{2\times 1} & J_{2\times 1} & 2J_{2\times 1} \\
 2J_{t\times2} & 2(J_{t}-I_t) & J_{t\times 1} & 2J_{t\times 1} & 2J_{t\times 1}\\
 J_{1\times 2} & J_{1\times t} & 0 & 1 & 1 \\
 J_{1\times 2} & 2J_{1\times t}  & 1 & 0 & 1\\
 2J_{1\times 2} & 2J_{1\times t}  & 1 & 1 & 0
\end{array}
\right).
\]

The first two rows of $-I_{t+5}-D(G)$ are equal, implying that $-1$ is a  distance  eigenvalue of $G$ with multiplicity at least $1$, and in  $-2I_{t+5}-D(G)$, there are $t$ equal rows, implying  that $-2$ is a distance  eigenvalue of $G$ with multiplicity at least $t-1$.

Let $Q$ be the quotient matrix of $D(G)$ with respect to the above partition on $V(G)$. Then
\[Q=\left(
\begin{array}{ccccc}
 1 & 2t & 1 & 1 & 2 \\
 4 & 2(t-1) & 1 & 2 & 2 \\
 2 & t & 0 & 1 & 1 \\
 2 & 2t & 1 & 0 & 1 \\
 4 & 2t & 1 & 1 & 0
\end{array}
\right).\]
Note that the above partition  is equitable. Thus the spectrum of $Q$ is contained in the distance spectrum of $G$, see \cite[Lemma 2.3.1]{BH}.
Let $f(\lambda)=\det(\lambda I_5-Q)$. By a direct calculation,
\[
f(\lambda)=\lambda^5-(2t-1)\lambda^4-(15t+17)\lambda^3-(33t+49)\lambda^2-(23t+44)\lambda-5t-12.
 \]
Note that $f(+\infty)>0$,
$f(-0.5)\approx-0.094<0$, $f(-\frac{t+1}{2t})=\frac{1}{32t^5}(t^5+19t^4-142t^3+62t^2-3t+1)>0$ for large enough $t$, $f(-1)=-2t<0$, $f(-3)=10t-24>0$ for $t\ge 4$, $f(-\infty)<0$.
It follows that the second largest root $\lambda^{(2)}$ of $f(\lambda)=0$ satisfies
$-\frac{t+1}{2t}<\lambda^{(2)}<-\frac{1}{2}$ for large enough $t$.
By the above argument, $\lambda^{(2)}$ is the second largest distance eigenvalue of $G$, i.e., $\lambda_2(G)<-\frac{1}{2}$.
\end{proof}

\begin{Lemma}\label{b6} Let $G=K_{s}(2,1)$, where $s\ge 2$. Then $\lambda_2(G)<-\frac{1}{2}$.
\end{Lemma}

\begin{proof} Let $G=K_{s}(2,1)$. Then $|V(G)|=s+3$. Since $K_{s}(2,1)$ is a distance-preserving induced subgraph of $K_{s+1}(2,1)$, it follows from Lemma \ref{Interlacing} that $\lambda_2(K_{s}(2,1))\le \lambda_2(K_{s+1}(2,1))$. Then the sequence $\{\lambda_2(K_{s}(2,1)): s=2,3, \dots\}$ does not decrease with $s$. So it suffices to show that $\lambda_2(G)<-\frac{1}{2}$ for large enough $s$.

Let $I=\{w_1,w_2,w_3\}$ be the independent set of $K_{s}(2,1)$. We use $u$ to denote the only common neighbour of $w_1$ and $w_2$, $v$ denotes the neighbor of $w_3$ in $K_{s}(2,1)$. Then we may partition $V(G)$ as $V(G)=V_1\cup V_2\cup\{u\}\cup\{v\}\cup\{w_3\}$, where $V_1=K_s\setminus\{u,v\}$, $V_2=\{w_1,w_2\}$.
Under this partition, we have
\[D(G)=
\left(
\begin{array}{ccccc}
 J_{s-2}-I_{s-2}  & 2J_{(s-2)\times 2} & J_{(s-2)\times 1} & J_{(s-2)\times 1} & 2J_{(s-2)\times 1}\\
 2J_{2\times(s-2)} & 2(J_{2}-I_2)   & J_{2\times 1} & 2J_{2\times 1} & 3J_{2\times 1}\\
 J_{1\times (s-2)} & J_{1\times 2}  & 0 & 1 & 2 \\
 J_{1\times (s-2)} & 2J_{1\times 2} & 1 & 0 & 1\\
 2J_{1\times (s-2)}& 3J_{1\times 2} & 2 & 1 & 0
\end{array}
\right).
\]
It is easy to see that $-1$ is a distance  eigenvalue of $G$ with multiplicity  at least $s-3$ and $-2$ is a distance eigenvalue of $G$ with multiplicity at least $1$.

Let $Q$ be the quotient matrix of $D(G)$ respect the above partition on $V(G)$. Then
\[Q=\left(
\begin{array}{ccccc}
 s-3 & 4 & 1 & 1 & 2 \\
 2(s-2) & 2 & 1 & 2 & 3 \\
 s-2 & 2 & 0 & 1 & 2 \\
 s-2 & 4 & 1 & 0 & 1 \\
 2(s-2) & 6 & 2 & 1 & 0
\end{array}
\right).\]
As the above partition  is equitable, the eigenvalues of $Q$ are also distance eigenvalues  of $G$.
Let $f(\lambda)=\det(\lambda I_5-Q)$. Then
\[
f(\lambda)=\lambda^5-(s-1)\lambda^4-12(s+1)\lambda^3-(40s+2)\lambda^2-(37s-10)\lambda-10s+4.
\]
Note that $f(+\infty)>0$, $f(0)=-10s+4<0$, $f(-\frac{1}{2})=\frac{1-2s}{32}<0$, $f(-\frac{s+1}{2s})=\frac{1}{32s^5}(2s^6-89s^5+233s^4-170 s^3-44s^2+3s+1)>0$ for large enough $s$, $f(-1)=4-2s<0$ for $s\ge 3$, $f(-4)=2(5s-34)>0$ for $s\ge 7$, $f(-\infty)<0$. It follows that the second largest root $\lambda^{(2)}$ of $f(\lambda)=0$ satisfies $-\frac{s+1}{2s}<\lambda^{(2)}<-\frac{1}{2}$ for large enough $s$. Thus, $\lambda^{(2)}$ is the second largest distance eigenvalue of $G$, i.e., $\lambda_2(G)<-\frac{1}{2}$.
\end{proof}

For a graph $G$ and its subgraph $H$ and a vertex $v$ of $G$ outside $H$, let $N_H(v)=N_G(v)\cap V(H)$ and $d_H(v)=|N_H(v)|$.

We are now ready to prove Theorem \ref{SED}.

\begin{proof}[Proof of Theorem \ref{SED}] Let $G$ be a split graph of order $n=s+t$. Let $K_s$ be the maximum clique and $I$ be the independent set in $G$ of size $t$. Then $s\ge 2$.

If $I=\emptyset$, then $G\cong K_s=K_s(0)$ and $\lambda_2(G)=-1<-\frac{1}{2}$.

Suppose that $I\neq\emptyset$.

Suppose first that $s=2$. Let $V(K_2)=\{u,v\}$. Assume that $d_G(u)\ge d_G(v)$. Since $F_{2}$ is a forbidden subgraph, we have $d_G(v)=1,2$. If $d_G(v)=2$, then $d_G(u)\le 3$ due to $F_1$ being forbidden. Thus $G\cong K_{2}(1,1)$ or $K_{2}(2,1)$. If $d_G(v)=1$, then $G\cong K_2(t)$.

By a direct calculation, we have $\lambda_2(K_{2}(2,1))\approx-0.5120<-\frac{1}{2}$. From Lemma \ref{block}, we have $\lambda_2(K_{2}(1,1))<-\frac{1}{2}$ and $\lambda_2(K_2(t))<-\frac{1}{2}$.

Suppose next that $s\ge 3$. Let $ID_G=\{z\in I: d_{K_s}(z)\ge 2\}$.

\noindent
{\bf Claim.} $|ID_G|=0,1$.

Otherwise, there exist $z_1$ and $z_2$  in $I$ with $d_{K_s}(z_1)\ge 2$ and $d_{K_s}(z_2)\ge 2$.  Let $x_1,x_2\in N_{K_s}(z_1)$ and $y_1,y_2\in N_{K_s}(z_2)$. There are three possibilities.

Suppose that $|\{x_1,x_2\}\cap\{y_1,y_2\}|=2$, i.e., $\{x_1,x_2\}=\{y_1,y_2\}$.
Note that $s\ge 3$. Since $K_s$ is a maximum clique, there exist $x_3,x_4\in V(K_s)$ such that $x_3z_1\notin E(G)$, $x_4z_2\notin E(G)$. Suppose that $x_3\neq x_4$. Then $x_3z_2\notin E(G)$. Otherwise, $G[\{x_1,x_2,x_3,x_4,z_2\}]\cong K_5-e$, and $\lambda_2(K_5-e)\approx-0.4495>-\frac{1}{2}$, a contradiction. So either $x_3\neq x_4$ and $x_3z_2\notin E(G)$ or $x_3=x_4$. In either case, $G[\{x_1,x_2,x_3,z_1,z_2\}]\cong F_{12}$, also a contradiction.

Suppose that $|\{x_1,x_2\}\cap\{y_1,y_2\}|=1$. Without loss of generality, let $y_1=x_1, y_2\neq x_2$. Then $G[\{x_1, x_2, y_2, z_1, z_2\}]\cong F_{11}$, a contradiction.

Suppose that $|\{x_1,x_2\}\cap\{y_1,y_2\}|=0$. Then  $G[\{x_1, x_2, y_1, y_2, z_1, z_2\}]\cong F_{13}$, also a contradiction.

Therefore,  the claim follows.

By the above claim, $|ID_G|=0,1$.

\noindent \textbf{Case 1.} $|ID_G|=1$.

Let $ID_G=\{z\}$.

If $s=3$, then $G$ is a $\theta$-bicyclic graph, so by Lemma \ref{bi2}, we have $\lambda_2(G)\le -\frac{1}{2}$ if and only if $G\cong B_6, B_7, B^{\theta}_k$.

Assume that $s\ge 4$. Then $d_{K_s}(z)=2$. Otherwise, there is a distance-preserving subgraph isomorphic to $K_5-e$ and $\lambda_2(K_5-e)\approx-0.4495>-\frac{1}{2}$, a contradiction.
Suppose that $s\ge 5$. Then there is a distance-preserving subgraph, say $H$, induced by $V(K_5)\cup \{z\}$. By a direct calculation, we have $\lambda_2(H)\approx-0.4913>-\frac{1}{2}$, a contradiction.
It thus follows that $s=4$. Let $N_{K_s}(z)=\{u,v\}$ with $d_G(u)\ge d_G(v)$. Let $w\in V(K_4)\setminus\{u,v\}$. Since $F_9$ is a forbidden graph, $d_G(w)=3$. Since $F_2$ is a forbidden subgraph, we have $d_G(v)=4, 5$. If $d_G(v)=5$, then, since $F_1$ is a forbidden subgraph, we have
 $d_G(u)=5, 6$. So $G\cong SP_1$ if $d_G(u)=6$ and $G\cong SP_2$ if $d_G(u)=5$. If $d_G(v)=4$, then $G\cong SP^t$ with $t\ge 0$.

By a direct calculation, we have $\lambda_2(SP_1)\approx-0.5106<-\frac{1}{2}$. Since $SP_2$ is an induced distance-preserving subgraph of $SP_1$, $\lambda_2(SP_2)\le \lambda_2(SP_1)<-\frac{1}{2}$. By Lemma \ref{b5}, $\lambda_2(SP^t)<-\frac{1}{2}$.

\noindent \textbf{Case 2.} $|ID_G|=0$.

If there is exactly one vertex in $K_s$ with degree not less than $s$, then $G\cong K_s(t)$, which is a block star, and by Lemma \ref{block}, $\lambda_2(K_s(t))<-\frac{1}{2}$.

Suppose that there are exactly two vertices in $K_s$, say $u, v$ with $d_G(u)\ge d_G(v)\ge s$. Since $F_2$ is a forbidden subgraph, we have $d_G(v)\le s$. So $d_G(v)=s$. Furthermore, since $F_1$ is a forbidden subgraph, we have $d_G(u)\le s+1$. Then $G\cong K_{s}(2,1)$ if $d_G(u)=s+1$, and $G\cong K_{s}(1,1)$ if $d_G(u)=s$. From Lemma \ref{b6}, $\lambda_2(K_{s}(2,1))<-\frac{1}{2}$. Since $K_{s}(1,1)$ is an induced distance-preserving subgraph of $K_{s}(2,1)$, $\lambda_2(K_{s}(1,1))\le \lambda_2(K_{s}(2,1))<-\frac{1}{2}$.

If there are $t$ vertices in $K_s$ with degree not less than $s$ with $t\ge 3$, then, since $F_6$ is a forbidden subgraph, we have $G\cong K_{s}(\underbrace{1,\dots, 1}_t)$,  which is a loose block. From Lemma \ref{block}, we have $\lambda_2(K_{s}(\underbrace{1,\dots, 1}_t))<-\frac{1}{2}$.
\end{proof}

\section{Conclusions}

In this article, we consider graphs for which the second largest distance eigenvalue is less than
$-\frac{1}{2}$. We show that  any such  graph must be chordal and characterize all bicyclic graphs and split graphs with the desired property. Characterizing  graphs with the desired property among graphs with forbidden subgraphs or minors would be a way to have a fuller understanding the behavior of the second largest distance eigenvalue.

\vspace{4mm}
\noindent
{\bf Declaration of competing interest}

The authors declare that they have no conflict of interest.

\vspace{5mm}

\noindent
{\bf Data availability}

No data was used for the research described in the article.

\vspace{5mm}

\noindent {\bf Acknowledgements}

We thank the two referees for constructive comments and suggestions.
This work was supported by the National Natural Science Foundation of China (Nos. 12101249 and  12071158).

\end{document}